\documentclass[12pt]{amsart}
\usepackage{latexsym,amsmath,amssymb,epsfig,amsthm}
\usepackage{graphicx}
\usepackage{tikz}
\usepackage[enableskew,vcentermath]{youngtab}
\usepackage{hyperref}
\hypersetup{colorlinks=false}
\input epsf
\newtheorem{theorem}{Theorem}[section]
\newtheorem{proposition}[theorem]{Proposition}
\newtheorem{corollary}[theorem]{Corollary}

\newtheorem{example}[theorem]{Example}
\newtheorem{lemma}[theorem]{Lemma}
\newtheorem{remark}[theorem]{Remark}

\newcommand{\Asc}{{\rm Asc}}
\newcommand{\ch}{{\rm ch}}

\newcommand{\Des}{{\rm Des}}

\newcommand{\sDes}{{\rm sDes}}
\newcommand{\stab}{{\rm Stab}}
\newcommand{\SYT}{{\rm SYT}}

\newcommand{\bB}{{\mathcal B}}

\newcommand{\pP}{{\mathcal P}}

\newcommand{\kk}{{\mathbf k}}
\newcommand{\bx}{{\mathbf x}}
\newcommand{\by}{{\mathbf y}}
\newcommand{\RR}{{\mathbb R}}
\newcommand{\fS}{{\mathfrak S}}
\newcommand{\NN}{{\mathbb N}}
\newcommand{\ZZ}{{\mathbb Z}}
\newcommand{\CC}{{\mathbb C}}
\renewcommand{\to}{\rightarrow}

\newcommand{\sm}{{\smallsetminus}}
\begin{document}
\title[Rees products of posets and equivariant 
gamma-positivity]
{Some applications of Rees products of posets to 
equivariant gamma-positivity}

\author{Christos~A.~Athanasiadis}

\address{Department of Mathematics \\
National and Kapodistrian University of Athens\\
Panepistimioupolis\\
15784 Athens, Greece}
\email{caath@math.uoa.gr}

\date{May 1, 2019}

\begin{abstract}
The Rees product of partially ordered sets was 
introduced by Bj\"orner and Welker. Using the theory 
of lexicographic shellability, Linusson, Shareshian 
and Wachs proved formulas, of significance in the theory 
of gamma-positivity, for the dimension of the homology
of the Rees product of a graded poset $P$ with a certain 
$t$-analogue of the chain of the same length as $P$. 
Equivariant generalizations of these formulas are 
proven in this paper, when a group of automorphisms
acts on $P$, and are applied to establish the 
Schur gamma-positivity of certain symmetric functions 
arising in algebraic and geometric combinatorics.
\end{abstract}

\maketitle

\section{Introduction}
\label{sec:intro}

The Rees product $P \ast Q$ of two partially ordered 
sets (posets, for short) was introduced and studied by 
Bj\"orner and Welker~\cite{BW05} as a combinatorial 
analogue of the Rees construction in commutative 
algebra (a precise definition of $P \ast Q$ can be 
found in Section~\ref{sec:pre}). The connection of 
the Rees product of posets to enumerative 
combinatorics was hinted in~\cite[Section~5]{BW05}, 
where it was conjectured that the dimension of the 
homology of the Rees product of the truncated 
Boolean algebra $B_n \sm \{ \varnothing \}$ of rank 
$n-1$ with an $n$-element chain equals the number of 
permutations of $[n] := \{1, 2,\dots,n\}$ without 
fixed points. This statement was generalized in 
several ways in~\cite{SW09}, using enumerative and 
representation theoretic methods, and 
in~\cite{LSW12}, using the theory of lexicographic 
shellability. 

One of the results of~\cite{LSW12} proves formulas
\cite[Corollary~3.8]{LSW12} for the dimension of the 
homology of the Rees product of an EL-shellable poset 
$P$ with a contractible poset which generalizes 
the chain of the same rank as $P$. This paper provides 
an equivariant analogue of this result which seems 
to have enough applications on its own to be of 
independent interest. To state it, let $P$ be a 
finite bounded poset, with minimum element $\hat{0}$ 
and maximum element $\hat{1}$, which is graded of 
rank $n+1$, with rank function $\rho: P \to 
\{0, 1,\dots,n+1\}$ (for basic terminology on posets, 
see \cite[Chapter~3]{StaEC1}). Fix a field $\kk$ and 
let $G$ be a finite group which 
acts on $P$ by order preserving bijections. Then, $G$ 
defines a permutation representation $\alpha_P(S)$ 
over $\kk$ for every $S \subseteq [n]$, induced by 
the action of $G$ on the set of maximal chains of 
the rank-selected subposet 
\begin{equation} \label{eq:P_S}
P_S \ = \ \{ x \in P: \rho(x) \in S \} \, \cup \,
\{\hat{0}, \hat{1} \}
\end{equation}
of $P$. Following Stanley's seminal work~\cite{Sta82}, 
we may consider the virtual $G$-representation
\begin{equation} \label{eq:beta(S)1}
\beta_P(S) \ = \ \sum_{T \subseteq S} (-1)^{|S-T|}
\, \alpha_P(T),
\end{equation}
defined equivalently by the equations
\begin{equation} \label{eq:beta(S)2}
\alpha_P(T) \ = \ \sum_{S \subseteq T} \beta_P(S)
\end{equation}
for $T \subseteq [n]$. The dimensions of $\alpha_P(S)$ 
and $\beta_P(S)$ are important enumerative invariants 
of $P$, known as the entries of its flag $f$-vector and
flag $h$-vector, respectively. When $P$ is Cohen--Macaulay 
over $\kk$, $\beta_P(S)$ is isomorphic to the 
non-virtual $G$-representation $\widetilde{H}_{|S|-1} 
(\bar{P}_S; \kk)$ induced on the top homology group of 
$\bar{P}_S := P_S \sm \{\hat{0}, \hat{1}\}$
\cite[Theorem~1.2]{Sta82}. As discussed and illustrated 
in various situations in~\cite{Sta82}, the decomposition 
of $\beta_P(S)$ as a direct sum of 
irreducible $G$-representations often leads to very 
interesting refinements of the flag $h$-vector of $P$.

As in references~\cite{LSW12, SW09}, we write 
$\beta(\bar{P})$ in place of $\beta_P([n])$ and note,
as just mentioned, that this $G$-representation is 
isomorphic to $\widetilde{H}_{n-1} (\bar{P}; \kk)$ if 
$P$ is Cohen--Macaulay over $\kk$. We 
denote by $T_{t,n}$ the poset whose Hasse diagram 
is a complete $t$-ary tree of height $n$, rooted at 
the minimum element. We denote by $P^-$, $P_{-}$ 
and $\bar{P}$ the poset obtained from $P$ by 
removing its maximum element, or minimum element, 
or both, respectively, and recall 
from~\cite{SW09} (see also Section~\ref{sec:pre}) 
that the action of $G$ on $P$ induces actions on
$P^- \ast T_{t,n}$ and $\bar{P} \ast T_{t,n-1}$ 
as well. We also write $[a, b] := \{a, 
a+1,\dots,b\}$ for integers $a \le b$ and denote 
by $\stab (\Theta)$ the set of all subsets, called 
stable, of $\Theta \subseteq \ZZ$ which do not 
contain two consecutive integers. The following 
result reduces to~\cite[Corollary~3.8]{LSW12}, 
proven in~\cite{LSW12} under additional 
shellability assumptions on $P$, in the special 
case of a trivial action.

\begin{theorem} \label{thm:main}
Let $G$ be a finite group acting on a finite bounded 
graded poset $P$ of rank $n+1$ by order preserving 
bijections. Then, 
  \begin{eqnarray} \label{eq:main1}
    \beta ((P^- \ast T_{t,n})_{-}) 
  & \cong_G & 
    \sum_{S \in \stab ([n-1])}
    \beta_P ([n] \sm S) \ t^{|S|} (1+t)^{n-2|S|} 
    \ \ + \\
  & & \nonumber \\
  & & \sum_{S \in \stab ([n-2])} \beta_P 
    ([n-1] \sm S) \ t^{|S|+1} (1+t)^{n-1-2|S|} 
    \nonumber
  \end{eqnarray}

\noindent
and 
  \begin{eqnarray} \label{eq:main2}
    \beta (\bar{P} \ast T_{t,n-1}) 
  & \cong_G & 
    \sum_{S \in \stab ([2, n-2])} \beta_P 
    ([n-1] \sm S) \ t^{|S|+1} (1+t)^{n-2-2|S|} 
    \ \ + \\
  & & \nonumber \\
  & & \sum_{S \in \stab ([2, n-1])}
    \beta_P ([n] \sm S) \ t^{|S|} (1+t)^{n-1-2|S|} 
    \nonumber
  \end{eqnarray}

\noindent
for every positive integer $t$, where $\cong_G$ 
stands for isomorphism of $G$-representations.
If $P$ is Cohen--Macaulay over $\kk$, then the 
left-hand sides of (\ref{eq:main1}) and (\ref{eq:main2}) 
may be replaced by the $G$-representations 
$\widetilde{H}_{n-1} ((P^- \ast T_{t,n})_{-}; \kk)$ 
and $\widetilde{H}_{n-1} (\bar{P} \ast T_{t,n-1}; 
\kk)$, respectively, and all representations which 
appear in these formulas are non-virtual.
\end{theorem}

Several applications of~\cite[Corollary~3.8]{LSW12} 
to $\gamma$-positivity appear in \cite{LSW12} 
\cite[Section~3]{Ath14} and are summarized in 
\cite[Section~2.4]{Ath17}. Theorem~\ref{thm:main} 
has non-trivial applications to Schur 
$\gamma$-positivity, which we now briefly discuss. 
A polynomial in $t$ with real coefficients is said 
to be \emph{$\gamma$-positive} if for some $m \in 
\NN$, it can be written as a nonnegative linear 
combination of the binomials $t^i (1+t)^{m-2i}$ for 
$0 \le i \le m/2$. Clearly, all such polynomials 
have symmetric and unimodal coefficients. Two 
symmetric function identities due to Gessel 
(unpublished), stated without proof 
in~\cite[Section~4]{LSW12} \cite[Theorem~7.3]{SW10}, 
can be written in the form 
\begin{equation} \label{eq:Ge2}
{\displaystyle 
\frac{1-t} {E(\bx; tz) - tE(\bx; z)} \ = \ 
1 \, + \, \sum_{n \ge 2} z^n 
\sum_{k=0}^{\lfloor (n-2)/2 \rfloor} 
\xi_{n,k}(\bx) \, t^{k+1} (1 + t)^{n-2k-2}}
\end{equation}
and
\begin{equation} \label{eq:Ge1}
{\displaystyle 
\frac{(1-t) E(\bx; tz)} {E(\bx; tz) - tE(\bx; z)} 
\ = \ 1 \, + \, \sum_{n \ge 1} z^n 
\sum_{k = 0}^{\lfloor (n-1)/2 \rfloor} 
\gamma_{n,k}(\bx) \, t^{k+1} (1 + t)^{n-1-2k}},
\end{equation}

\smallskip
\noindent
where $E(\bx; z) = \sum_{n \ge 0} e_n(\bx) z^n$ 
is the generating function for the elementary 
symmetric functions in $\bx = (x_1, x_2,\dots)$ 
and the $\xi_{n,k}(\bx)$ and $\gamma_{n,k}(\bx)$ 
are Schur-positive symmetric functions, whose 
coefficients in the Schur basis can be explicitly
described (see Corollary~\ref{cor:Gessel}). The 
coefficients of $z^n$ in the right-hand sides of 
Equations~(\ref{eq:Ge2}) and~(\ref{eq:Ge1}) are 
Schur $\gamma$-positive symmetric functions, in 
the sense that their coefficients in the Schur 
basis are $\gamma$-positive polynomials in $t$
with the same center of symmetry. Their Schur 
$\gamma$-positivity refines the $\gamma$-positivity
of derangement and Eulerian polynomials, respectively;
see \cite[Section~2.5]{Ath17} for more information.

We will show (see Section~\ref{sec:app}) that 
Gessel's identities 
can in fact be derived from the special case of 
Theorem~\ref{thm:main} in which $P^-$ is the 
Boolean algebra $B_n$, endowed with the natural 
symmetric group action. Moreover, applying the 
theorem when $P^-$ is a natural signed analogue 
of $B_n$, endowed with a hyperoctahedral group 
action, we obtain new identities of the form 
\begin{equation} \label{eq:Ath2+}
{\displaystyle \frac{E(\bx; tz) - t E(\bx; z)} 
{E(\bx; tz) E(\by; tz) - tE(\bx; z) E(\by; z)} 
\ = \ \sum_{n \ge 0} z^n \,
\sum_{k=0}^{\lfloor n/2 \rfloor} 
\xi^+_{n,k}(\bx, \by) \, t^k (1 + t)^{n-2k}},
\end{equation}
\begin{equation} \label{eq:Ath2-}
{\displaystyle \frac{E(\bx; z) - E(\bx; tz)} 
{E(\bx; tz) E(\by; tz) - tE(\bx; z) E(\by; z)} 
\ = \ \sum_{n \ge 1} z^n
\sum_{k=0}^{\lfloor (n-1)/2 \rfloor} 
\xi^{-}_{n,k}(\bx, \by) \, t^k (1 + t)^{n-1-2k}}, 
\end{equation}
\begin{equation} \label{eq:Ath1+}
{\displaystyle \frac{E(\bx; z) E(\bx; tz) 
\left( E(\by; tz) - t E(\by; z) \right)} 
{E(\bx; tz) E(\by; tz) - tE(\bx; z) E(\by; z)} 
\ = \ \sum_{n \ge 0} z^n 
\sum_{k=0}^{\lfloor n/2 \rfloor} \gamma^{+}_{n,k}
(\bx, \by) \, t^k (1 + t)^{n-2k}} 
\end{equation}
and
\begin{equation} \label{eq:Ath1-}
{\displaystyle \frac{t E(\bx; z) E(\bx; tz) 
\left( E(\by; z) - E(\by; tz) \right)} 
{E(\bx; tz) E(\by; tz) - tE(\bx; z) E(\by; z)} \ = \
\sum_{n \ge 1} z^n \sum_{k=1}^{\lfloor (n+1)/2 \rfloor} 
\gamma^{-}_{n,k}(\bx, \by) \, t^k (1 + t)^{n+1-2k}}, 
\end{equation}

\smallskip
\noindent
where the $\xi^{\pm}_{n,k}(\bx, \by)$ and 
$\gamma^{\pm}_{n,k}(\bx, \by)$ are Schur-positive 
symmetric functions in the sets of variables $\bx = 
(x_1, x_2,\dots)$ and $\by = (y_1, y_2,\dots)$ 
separately. Their Schur positivity refines 
the $\gamma$-positivity of type $B$ analogues or 
variants of derangement and Eulerian polynomials; 
this is explained and 
generalized in the sequel \cite{Ath1x} to this 
paper. Note that for $\bx = 0$, the left-hand side 
of Equation~(\ref{eq:Ath2+}) specializes to that 
of~(\ref{eq:Ge2}) (with $\bx$ replaced by $\by$) 
and the sum of the left-hand sides of 
Equations~(\ref{eq:Ath1+}) and ~(\ref{eq:Ath1-})
specializes to that of~(\ref{eq:Ge1}) (again with 
$\bx$ replaced by $\by$). 

Various combinatorial and algebraic-geometric 
interpretations of the left-hand sides of 
Equations~(\ref{eq:Ge2}) and~(\ref{eq:Ge1}) are 
discussed in \cite[Section~7]{SW10} 
\cite[Section~4]{LSW12} \cite{SW16b}. For instance,
by \cite[Proposition~4.20]{Sta92}, the coefficient
of $z^n$ in the left-hand side of~(\ref{eq:Ge2}) 
can be interpreted as the Frobenius characteristic 
of the symmetric group representation on the local 
face module of the barycentric subdivision of the 
$(n-1)$-dimensional simplex, twisted by the sign 
representation. Thus, the Schur $\gamma$-positivity
of this coefficient, manifested by 
Equation~(\ref{eq:Ge2}), is an instance of the 
local equivariant Gal phenomenon, as discussed in 
\cite[Section~5.2]{Ath17}. Section~\ref{sec:face} 
shows that another instance of this phenomenon 
follows from the specialization $\bx = \by$ of 
Equation~(\ref{eq:Ath2+}). Similarly, setting 
$\by=0$ to (\ref{eq:Ath1+}) yields another 
identity, recently proven by Shareshian and Wachs 
(see Proposition~3.3 and Theorem~3.4 in~\cite{SW17}) 
in order to establish the equivariant Gal phenomenon 
for the symmetric group action on the $n$-dimensional 
stellohedron and Section~\ref{sec:toric} combines 
Equation~(\ref{eq:Ge2}) with~(\ref{eq:Ath1+}) to 
establish the same phenomenon for the hyperoctahedral
group action on its associated Coxeter complex. Further 
applications of Theorem~\ref{thm:main} are given 
in~\cite{Ath1x}. It would be interesting to find 
direct combinatorial proofs of 
Equations~(\ref{eq:Ath2+}) -- (\ref{eq:Ath1-}) and 
to generalize other known interpretations of the 
left-hand sides of Equations~(\ref{eq:Ge2}) 
and~(\ref{eq:Ge1}) to those of~(\ref{eq:Ath2+}) -- 
(\ref{eq:Ath1-}).

\medskip
\noindent
\textbf{Outline.} The proof of Theorem~\ref{thm:main} 
is given in Section~\ref{sec:proof}, after the 
relevant background and definitions are explained in 
Section~\ref{sec:pre}. This proof is fairly 
elementary and different from that 
of~\cite[Corollary~3.8]{LSW12}. Section~\ref{sec:app} 
derives Equations~(\ref{eq:Ge2}) -- (\ref{eq:Ath1-}) 
from Theorem~\ref{thm:main} and provides explicit 
combinatorial interpretations, in terms of standard 
Young (bi)tableaux and their descents, for the 
Schur-positive symmetric functions which appear 
there. Sections~\ref{sec:face} and~\ref{sec:toric} 
provide the promised applications of 
Equations~(\ref{eq:Ath2+}) and~(\ref{eq:Ath1+}) to 
the equivariant $\gamma$-positivity of the symmetric 
group representation on the local face module of 
a certain triangulation of the simplex and the 
hyperoctahedral group representation on the cohomology
of the projective toric variety associated to the 
Coxeter complex of type $B$.

\section{Preliminaries}
\label{sec:pre}

This section briefly records definitions and 
background on posets, group representations and 
(quasi)symmetric functions. For basic notions and 
more information on these topics, the reader is 
referred to the sources
\cite{Sa01} \cite{Sta82} \cite[Chapter~3]{StaEC1} 
\cite[Chapter~7]{StaEC2} \cite{Wa07}. The 
symmetric group of permutations of the set $[n]$ 
is denoted by $\fS_n$ and the cardinality of a 
finite set $S$ by $|S|$.

\medskip
\noindent
\textbf{Group actions on posets and Rees products.}
All groups and posets considered here are assumed 
to be finite. Homological notions for posets always 
refer to those of their order complex; 
see~\cite[Lecture~1]{Wa07}. A poset $P$ has the 
structure of a \emph{$G$-poset} 
if the group $G$ acts on $P$ by order preserving 
bijections. Then, $G$ induces a representation on 
every reduced homology group $\widetilde{H}_i (P; 
\kk)$, for every field $\kk$. 

Suppose that $P$ is a $G$-poset with minimum 
element $\hat{0}$ and maximum element $\hat{1}$. 
Sundaram~\cite{Su94} (see also 
\cite[Theorem~4.4.1]{Wa07}) established the 
isomorphism of $G$-representations
\begin{equation} \label{eq:sun}
  \bigoplus_{k \ge 0} \ (-1)^k \bigoplus_{x \in P/G} 
   \widetilde{H}_{k-2} ((\hat{0}, x); \kk)
   \uparrow^G_{G_x} \ \, \cong_G \ 0.
\end{equation}

\noindent
Here $P/G$ stands for a complete set of $G$-orbit 
representatives, $(\hat{0}, x)$ denotes the open 
interval of elements of $P$ lying strictly between 
$\hat{0}$ and $x$, $G_x$ is the stabilizer of 
$x$ and $\uparrow$ denotes induction. Moreover,
$\widetilde{H}_{k-2} ((\hat{0}, x); \kk)$ is 
understood to be the trivial representation 
$1_{G_x}$ if $x = \hat{0}$ and $k=0$, or $x$ 
covers $\hat{0}$ and $k = 1$. 

The \emph{Lefschetz character} of a finite 
$G$-poset $P$ (over the field $\kk$) is defined 
as the virtual $G$-representation

\[ L(P; G) \ := \ \bigoplus_{k \ge 0} \, (-1)^k \, 
                  \widetilde{H}_k (P; \kk). \]
Note that $L(P; G) = (-1)^r \widetilde{H}_r 
(P; \kk)$, if $P$ is Cohen--Macaulay over $\kk$ 
of rank $r$. 

Given finite graded posets $P$ and $Q$ with rank 
functions $\rho_P$ and $\rho_Q$, respectively, 
their \emph{Rees product} is defined
in~\cite{BW05} as $P \ast Q = \{ (p, q) \in P 
\times Q: \rho_P (p) \ge \rho_Q (q) \}$, with 
partial order defined by setting $(p_1, q_1) 
\preceq (p_2, q_2)$ if all of the following 
conditions are satisfied:

\begin{itemize}
\itemsep=0pt
\item[$\bullet$]
$p_1 \preceq p_2$ holds in $P$,

\item[$\bullet$]
$q_1 \preceq q_2$ holds in $Q$ and

\item[$\bullet$]
$\rho_P (p_2) - \rho_P (p_1) \ge \rho_Q 
(q_2) - \rho_Q (q_1)$.
\end{itemize}

\noindent
Equivalently, $(p_1, q_1)$ is covered by $(p_2, q_2)$ 
in $P \ast Q$ if and only if (a) $p_1$ is covered by 
$p_2$ in $P$; and (b) either $q_1 = q_2$, or $q_1$ is 
covered by $q_2$ in $Q$. We note that the Rees product
$P \ast Q$ is graded with rank function given by 
$\rho(p,q) = \rho_P(p)$ for $(p,q) \in P \ast Q$, and 
that if $P$ is a $G$-poset, then so is $P \ast Q$ 
with the $G$-action defined by setting $g \cdot (p, 
q) = (g \cdot p, q)$ for $g \in G$ and $(p, q) \in 
P \ast Q$.

A bounded graded $G$-poset $P$, with maximum element
$\hat{1}$, is said to be \emph{$G$-uniform} 
\cite[Section~3]{SW09} if the following hold:

\begin{itemize}
\itemsep=0pt
\item[$\bullet$]
the intervals $[x, \hat{1}]$ and $[y, \hat{1}]$ in 
$P$ are isomorphic for all $x, y \in P$ of the same 
rank,

\item[$\bullet$]
the stabilizers $G_x$ and $G_y$ are isomorphic for 
all $x, y \in P$ of the same rank, and

\item[$\bullet$]
there is an isomorphism between $[x, \hat{1}]$ and 
$[y, \hat{1}]$ that intertwines the actions of $G_x$ 
and $G_y$, for all $x, y \in P$ of the same rank.
\end{itemize}

Given a sequence of groups $G = (G_0, 
G_1,\dots,G_n)$, a sequence of posets $(P_0, 
P_1,\dots,P_n)$ is said to be \emph{$G$-uniform} 
\cite[Section~3]{SW09} if 

\begin{itemize}
\itemsep=0pt
\item[$\bullet$]
$P_k$ is $G_k$-uniform of rank $k$ for all $k$,

\item[$\bullet$]
$G_k$ is isomorphic to the stabilizer $(G_n)_x$
for every $x \in P_n$ of rank $n-k$ and all $k$, 
and

\item[$\bullet$]
there is an isomorphism between $P_k$ and the 
interval $[x, \hat{1}]$ in $P_n$ that intertwines 
the actions of $G_k$ and $(G_n)_x$ for every $x 
\in P_n$ of rank $n-k$ and all $k$.
\end{itemize}

\noindent
Under these assumptions, Shareshian and Wachs 
\cite[Proposition~3.7]{SW09} established the 
isomorphism of $G$-representations
\begin{equation} \label{eq:uniform}
  1_{G_n} \oplus \, \bigoplus_{k=0}^n \, 
  W_k(P_n;G_n) \, [k+1]_t \, L((P_{n-k} \ast 
  T_{t,n-k})_{-}; G_{n-k}) \uparrow^{G_n}_{G_{n-k}} 
  \ \, \cong_G \ 0
\end{equation}

\noindent 
for every positive integer $t$, where $W_k(P_n;G_n)$ 
is the number of $G_n$-orbits of elements of $P_n$ 
of rank $k$ and $[k+1]_t := 1 + t + \cdots + t^k$.

\begin{example} \label{ex:uniform} \rm
The Boolean algebra $B_n$ consists of all subsets 
of $[n]$, partially ordered by inclusion. When endowed 
with the standard $\fS_n$-action, it becomes a 
prototypical example of an $\fS_n$-uniform poset. 
Every element $x \in B_n$ of rank $k$ is a set of 
cardinality $k$; its stabilizer $(\fS_n)_x = \{w \in 
\fS_n: w x = x\}$ is isomorphic to the Young 
subgroup $\fS_k \times \fS_{n-k}$ of $\fS_n$, which 
can be defined as the stalilizer of $[k]$. The sequence 
$(B_0, B_1,\dots,B_n)$ can easily be verified to be 
$(G_0, G_1,\dots,G_n)$-uniform for $G_k := \fS_k 
\times \fS_{n-k}$. 
\qed
\end{example}

\noindent
\textbf{Permutations, Young tableaux and symmetric
functions.} Our notation concerning these topics 
follows mostly that of \cite{Sa01} 
\cite[Chapter~1]{StaEC1} \cite[Chapter~7]{StaEC2}.
In particular, the set of standard Young tableaux 
of shape $\lambda$ is denoted by $\SYT(\lambda)$, 
the descent set $\{ i \in [n-1]: w(i) > w(i+1) \}$ 
of a permutation $w \in \fS_n$ is denoted by $\Des
(w)$ and that of a tableau $Q \in 
\SYT(\lambda)$, consisting of those entries $i$ for 
which $i+1$ appears in $Q$ in a lower row than $i$,
is denoted by $\Des(Q)$. We recall that the 
Robinson--Schensted correspondence is a 
bijection from the 
symmetric group $\fS_n$ to the set of pairs 
$(\pP, Q)$ of standard Young tableaux of the same 
shape and size $n$. This correspondence has the
property \cite[Lemma~7.23.1]{StaEC2} that $\Des(w) 
= \Des(Q(w))$ and $\Des(w^{-1}) = \Des(\pP(w))$, 
where $(\pP(w), Q(w))$ is the pair of tableaux 
associated to $w \in \fS_n$.

We will consider symmetric functions in the 
indeterminates $\bx = (x_1, x_2,\dots)$ over the 
complex field $\CC$. We denote by $E(\bx; z) := 
\sum_{n \ge 0} e_n(\bx) z^n$ and $H(\bx; z) := 
\sum_{n \ge 0} h_n(\bx) z^n$ the generating 
functions for the elementary and complete 
homogeneous symmetric functions, respectively, 
and recall the identity $E(\bx; z) H(\bx; -z) = 1$. 
The (Frobenius) characteristic map, a $\CC$-linear 
isomorphism of fundamental importance from the 
space of virtual $\fS_n$-representations to that 
of homogeneous symmetric functions of degree $n$,
will be denoted by $\ch$. This map sends the 
irreducible $\fS_n$-representation corresponding
to $\lambda \vdash n$ to the Schur function 
$s_\lambda(\bx)$ associated to $\lambda$ and thus
it sends non-virtual $\fS_n$-representations to 
Schur-positive symmetric functions.
The \emph{fundamental quasisymmetric function} 
associated to $S \subseteq [n-1]$ is defined as
  \begin{equation} \label{eq:defF(x)}
    F_{n, S} (\bx) \ =
    \sum_{\substack{1 \le i_1 \le i_2 \le \cdots 
    \le i_n \\ j \in S \,\Rightarrow\, 
    i_j < i_{j+1}}}
    x_{i_1} x_{i_2} \cdots x_{i_n}.
  \end{equation}
The following well known expansion 
\cite[Theorem~7.19.7]{StaEC2}
  \begin{equation} \label{eq:sFexpan}
    s_\lambda(\bx) \ = \ \sum_{Q \in \SYT (\lambda)} 
    F_{n, \Des(Q)} (\bx)
  \end{equation}
will be used in Section~\ref{sec:app}.

For the applications given there,
we need the analogues of these concepts in the 
representation theory of the hyperoctahedral group 
of signed permutations of the set $[n]$, denoted 
here by $\bB_n$. We will keep this discussion rather
brief and refer to \cite[Section~2]{AAER17} for 
more information.

A \emph{bipartition} of a positive integer $n$, 
written $(\lambda, \mu) \vdash n$, is any pair 
$(\lambda, \mu)$ of integer partitions of total 
sum $n$. A \emph{standard Young bitableau} of 
shape $(\lambda, \mu) \vdash n$ is any pair $Q 
= (Q^+, Q^{-})$ of Young tableaux such that 
$Q^+$ has shape $\lambda$, $Q^{-}$ has shape
$\mu$ and every element of $[n]$ appears
exactly once as an entry of $Q^+$ or $Q^{-}$.
The tableaux $Q^+$ and $Q^{-}$ are called the 
\emph{parts} of $Q$ and the number $n$ is its 
\emph{size}. The Robinson--Schensted 
correspondence of type $B$, as described 
in~\cite[Section~6]{Sta82} (see 
also~\cite[Section~5]{AAER17}), is a 
bijection from the group $\bB_n$ to the set 
of pairs $(\pP, Q)$ of standard Young bitableaux 
of the same shape and size $n$.
 
The (Frobenius) characteristic map for the 
hyperoctahedral group, denoted by $\ch_\bB$, 
is a $\CC$-linear isomorphism from the space of 
virtual $\bB_n$-representations to that of 
homogeneous symmetric functions of degree $n$ in 
the sets of indeterminates $\bx = (x_1, x_2,\dots)$ 
and $\by = (y_1, y_2,\dots)$ separately; see, for 
instance, \cite[Section~5]{Ste92}. The map $\ch_\bB$
sends the irreducible $\bB_n$-representation 
corresponding to $(\lambda, \mu) \vdash n$ to the 
function $s_\lambda(\bx) s_\mu(\by)$ and thus
it sends non-virtual $\bB_n$-representations to 
Schur-positive functions, meaning nonnegative 
integer linear combinations of the functions 
$s_\lambda(\bx) s_\mu(\by)$. The following basic 
properties of $\ch_\bB$ will be useful in 
Section~\ref{sec:app}:

\smallskip
\begin{itemize}
\itemsep=1pt
\item[$\bullet$]
$\ch_\bB(1_{\bB_n}) = h_n(\bx)$, where $1_{\bB_n}$
is the trivial $\bB_n$-representation,

\item[$\bullet$]
$\ch_\bB (\sigma \otimes \tau 
\uparrow^{\bB_n}_{\bB_k \times \bB_{n-k}}) = 
\ch_\bB(\sigma) \cdot \ch_\bB (\tau)$ for all 
representations $\sigma$ and $\tau$ of $\bB_k$ 
and $\bB_{n-k}$, respectively, where $\bB_k 
\times \bB_{n-k}$ is the Young subgroup of $\bB_n$
consisting of all signed permutations which preserve
the set $\{ \pm 1, \pm 2,\dots,\pm k\}$, 

\item[$\bullet$]
$\ch_\bB (\uparrow^{\bB_n}_{\fS_n} \rho) = \ch
(\rho) (\bx, \by)$ for every $\fS_n$-representation 
$\rho$, where $\fS_n \subset \bB_n$ is the standard
embedding.
\end{itemize}

\smallskip
We denote by $E(\bx, \by; z) := \sum_{n \ge 0} e_n
(\bx, \by) z^n$ and $H(\bx, \by; z) := \sum_{n \ge 0} 
h_n(\bx, \by) z^n$ the generating function for the 
elementary and complete homogeneous, respectively,
symmetric functions in the variables $(\bx, \by) = 
(x_1, x_2,\dots,y_1, y_2,\dots)$ and note that 
$E(\bx, \by; z) = E(\bx; z) E(\by; z)$, since $e_n
(\bx, \by) = \sum_{k=0}^n e_k(\bx) e_{n-k}(\by)$, 
and similarly that $H(\bx, \by; z) = H(\bx; z) 
H(\by; z)$.  

The \emph{signed descent set} 
\cite[Section~2]{AAER17} \cite{Poi98} of $w \in 
\bB_n$, denoted $\sDes(w)$, records the positions 
of the increasing (in absolute value) runs of 
constant sign in the sequence $(w(1), 
w(2),\dots,w(n))$. Formally, we may define 
$\sDes(w)$ as the pair $(\Des(w), \varepsilon)$, 
where $\varepsilon = (\varepsilon_1, 
\varepsilon_2,\dots,\varepsilon_n) \in \{-, +\}^n$ 
is the sign vector with $i$th coordinate equal to 
the sign of $w(i)$ and $\Des(w)$ consists of the 
indices $i \in [n-1]$ for which either $\varepsilon_i 
= +$ and $\varepsilon_{i+1} = -$, or else  
$\varepsilon_i = \varepsilon_{i+1}$ and $|w(i)| > 
|w(i+1)|$ (this definition is slightly different 
from, but equivalent to, the ones given in 
\cite{AAER17,Poi98}). The fundamental 
quasisymmetric function associated to $w$,
introduced by Poirier~\cite{Poi98} in a more 
general setting, is defined as 
 \begin{equation} \label{eq:defF(xy)}
    F_{\sDes(w)} (\bx, \by) \ =
    \sum_{\substack{i_1 \le i_2 \le \cdots 
    \le i_n \\ j \in \Des(w) \, 
    \Rightarrow \, i_j < i_{j+1}}}
    z_{i_1} z_{i_2} \cdots z_{i_n},
  \end{equation}
where $z_{i_j} = x_{i_j}$ if $\varepsilon_j = +$, and 
$z_{i_j} = y_{i_j}$ if $\varepsilon_j = -$. Given a 
standard Young bitableau $Q$ of size $n$, one defines 
the signed descent set $\sDes(Q)$ as the pair 
$(\Des(Q), \varepsilon)$, where $\varepsilon = 
(\varepsilon_1, \varepsilon_2,\dots,\varepsilon_n) 
\in \{-, +\}^n$ is the sign vector with $i$th 
coordinate equal to the sign of the part of $Q$ in 
which $i$ appears and $\Des(Q)$ is the set of indices 
$i \in [n-1]$ for which either $\varepsilon_i = +$ 
and $\varepsilon_{i+1} = -$, or else  $\varepsilon_i 
= \varepsilon_{i+1}$ and $i+1$ appears in $Q$ in a 
lower row than $i$. The function $F_{s\Des(Q)} (\bx, 
\by)$ is then defined by the right-hand side of 
Equation~(\ref{eq:defF(xy)}), with $w$ replaced by
$Q$; see~\cite[Section~2]{AAER17}. The analogue 
  \begin{equation} \label{eq:sxyFexpan}
    s_\lambda(\bx) s_\mu(\by) \ = \sum_{Q \in \SYT 
    (\lambda,\mu)} F_{s\Des(Q)} (\bx, \by)
  \end{equation}
of the expansion~(\ref{eq:sFexpan}) holds 
(\cite[Proposition~4.2]{AAER17}) and the 
Robinson--Schensted correspondence of type $B$ has 
the properties that $\sDes(w) = \sDes(Q^B(w))$ and 
$\sDes(w^{-1}) = \sDes(\pP^B(w))$, where $(\pP^B(w), 
Q^B(w))$ is the pair of bitableaux associated to $w 
\in \bB_n$; see~\cite[Proposition~5.1]{AAER17}.

\section{Proof of Theorem~\ref{thm:main}}
\label{sec:proof}

This section proves Theorem~\ref{thm:main} using 
only the definition of Rees product and the defining 
equation~(\ref{eq:beta(S)1}) of the representations 
$\beta_P(S)$. For $S = \{s_1, s_2,\dots,s_k\} 
\subseteq [n]$ with $s_1 < s_2 < \cdots < s_k$ we 
set 
  \begin{eqnarray*}
    \varphi_t(S) & := & [s_1 + 1]_t \, [s_2 - s_1 + 1]_t 
                        \cdots [s_k - s_{k-1} + 1]_t \\
    \psi_t(S) & := & [s_1]_t \, [s_2 - s_1 + 1]_t \cdots 
                        [s_k - s_{k-1} + 1]_t 
  \end{eqnarray*}

\medskip
\noindent
where, as mentioned already, $[m]_t = 1 + t + \cdots + 
t^{m-1}$ for positive integers $m$.

\begin{lemma} \label{lem:alpha}
Let $G$ be a finite group, $P$ be a finite bounded 
graded $G$-poset of rank $n+1$, as in 
Theorem~\ref{thm:main}, and $Q, R$ be the posets 
defined by $\bar{Q} = (P^- \ast T_{t,n})_{-}$ and 
$\bar{R} = \bar{P} \ast T_{t,n-1}$. Then, 
  \begin{eqnarray*}
    \alpha_Q (S) & \cong_G & \varphi_t(S) \, 
    \alpha_P (S) \\
    \alpha_R (S) & \cong_G & \psi_t(S) \, 
    \alpha_P (S)
  \end{eqnarray*}

\noindent
for all positive integers $t$ and $S \subseteq [n]$.
\end{lemma}

\begin{proof}
Let $S = \{s_1, s_2,\dots,s_k\} \subseteq [n]$ with 
$s_1 < s_2 < \cdots < s_k$ and let $\rho: T_{t,n} 
\to \NN$ be the rank function of $T_{t,n}$. By the 
definition of Rees product, the maximal chains in 
$Q_S$ have the form
\begin{equation} \label{Qchain}
   \hat{0} \, \prec \, (p_1, \tau_1) \, \prec \, 
   (p_2, \tau_2) \, \prec \, \cdots \, \prec \, 
   (p_k, \tau_k) \, \prec \, \hat{1}
\end{equation}
where $\hat{0} \prec p_1 \prec p_2 \prec \cdots 
\prec p_k \prec \hat{1}$ is a maximal chain in 
$P_S$ and $\tau_1 \preceq \tau_2 \preceq \cdots 
\preceq \tau_k$ is a multichain in $T_{t,n}$ such 
that 

\medskip
\begin{itemize}
\itemsep=0pt
\item[$\bullet$]
$0 \le \rho(\tau_1) \le s_1$ and

\item[$\bullet$]
$\rho(\tau_j) - \rho(\tau_{j-1}) \le s_j - s_{j-1}$ 
for $2 \le j \le k$.
\end{itemize}

\medskip
\noindent
Let $m_t(S)$ be the number of these multichains. 
Since the elements of $G$ act on~(\ref{Qchain}) by 
fixing the $\tau_j$ and acting on the corresponding 
maximal chain of $P_S$, we have $\alpha_Q(S) 
\cong_G m_t(S) \, \alpha_P(S)$. To choose such a 
multichain $\tau_1 \preceq \tau_2 \preceq \cdots 
\preceq \tau_k$, we need to specify the sequence 
$i_1 \le i_2 \le \cdots \le i_k$ of ranks of its 
elements so that $i_j - i_{j-1} \le s_j - s_{j-1}$ 
for $1 \le j \le k$, where $i_0 := s_0 := 0$, and 
choose its maximum element $\tau_k$ in $t^{i_k}$ 
possible ways. Summing over all such sequences, 
we get

\[  m_t (S) \ = \ \sum_{(i_1, i_2,\dots,i_k)} 
    t^{i_k} \ = \ \sum_{0 \le a_j \le s_j - 
    s_{j-1}} t^{a_1 + a_2 + \cdots + a_k} \ = \ 
    \varphi_t (S) \] 

\medskip
\noindent
and the result for $\alpha_Q(S)$ follows. The same
argument applies to $\alpha_R(S)$; one simply has to 
switch the condition for the rank of $\tau_1$ to $0 
\le \rho(\tau_1) \le s_1 - 1$.  
\end{proof}

The proof of the following technical lemma will be 
given after that of Theorem~\ref{thm:main}.

\begin{lemma} \label{lem:phipsi}
For every $S \subseteq [n]$ we have
\begin{equation} \label{eq:chi1}
   \sum_{S \subseteq T \subseteq [n]} (-1)^{n-|T|} 
     \, \varphi_t(T) \ = \ \begin{cases} 
     0, & \text{if $[n] \sm S$ is not stable}, \\
     t^r(1+t)^{n-2r}, & \text{if $[n] \sm S$ is 
                        stable and $n \in S$}, \\
     t^r(1+t)^{n+1-2r}, & \text{if $[n] \sm S$ is 
                        stable and $n \not\in S$} 
   \end{cases} 
\end{equation}

\medskip
\noindent
and
\begin{equation} \label{eq:chi2}
   \sum_{S \subseteq T \subseteq [n]} (-1)^{n-|T|} 
     \, \psi_t(T) \ = \ \begin{cases} 
     0, & \text{if $1 \not\in S$}, \\
     0, & \text{if $[n] \sm S$ is not stable}, \\
     t^r(1+t)^{n-1-2r}, & \text{if $[n] \sm S$ is 
                        stable and $1, n \in S$}, \\
     t^r(1+t)^{n-2r}, & \text{if $[n] \sm S$ is 
                stable, $1 \in S$ and $n \not\in S$}, 
   \end{cases} 
\end{equation}

\smallskip
\noindent
where $r = n - |S|$.
\end{lemma}

\smallskip
\noindent
\emph{Proof of Theorem~\ref{thm:main}}. Using 
Equations~(\ref{eq:beta(S)1}) and~(\ref{eq:beta(S)2}), 
as well as Lemma~\ref{lem:alpha}, we compute that

  \begin{eqnarray*}
    \beta_Q ([n]) & = & \sum_{T \subseteq [n]} 
    (-1)^{n-|T|} \, \alpha_Q (T) \ \cong_G \ 
    \sum_{T \subseteq [n]} (-1)^{n-|T|} \, 
    \varphi_t(T) \, \alpha_P (T) \\
  & & \\
  & = & \sum_{T \subseteq [n]} (-1)^{n-|T|} \, 
    \varphi_t(T) \, \sum_{S \subseteq T} \beta_P(S)
  \\ & & \\
  & = & \sum_{S \subseteq [n]} \beta_P(S) 
    \sum_{S \subseteq T \subseteq [n]} (-1)^{n-|T|} 
    \, \varphi_t(T) 
  \end{eqnarray*}

\noindent
and find similarly that
\[ \beta_R ([n]) \ \cong_G \ \sum_{S \subseteq [n]} 
   \beta_P(S) \sum_{S \subseteq T \subseteq [n]} 
   (-1)^{n-|T|} \, \psi_t(T). \]

\smallskip
\noindent
The proof follows from these formulas and 
Lemma~\ref{lem:phipsi}. 
For the last statement of the theorem one has to note 
that, as a consequence of~\cite[Corollary~2]{BW05}, 
if $P$ is Cohen--Macaulay over $\kk$, then so are the 
Rees products $P^{-} \ast T_{t,n}$ and $\bar{P} \ast 
T_{t,n-1}$.
\qed

\medskip
\noindent
\emph{Proof of Lemma~\ref{lem:phipsi}}. Let us denote
by $\chi_t(S)$ the left-hand side of~(\ref{eq:chi1})
and write $S = \{s_1, s_2,\dots,s_k\}$, with $1 \le 
s_1 < s_2 < \cdots < s_k \le n$. By definition, we 
have 
\begin{equation} \label{eq:chi}
   \chi_t(S) \ = \ \chi_t(s_1) \, \chi_t(s_2 - s_1) 
   \cdots \chi_t(s_k - s_{k-1}) \, \omega_t(n - s_k),
\end{equation}
where
  \begin{eqnarray*} 
    \chi_t(n) & := & \sum_{n \in T \subseteq [n]} 
               (-1)^{n-|T|} \, \varphi_t(T) \\
    & & \\
    \omega_t(n) & := & \sum_{T \subseteq [n]} 
               (-1)^{n-|T|} \, \varphi_t(T) 
  \end{eqnarray*}

\smallskip
\noindent
for $n \ge 1$ and $\omega_t(0) := 1$. We claim that 

\[ \chi_t(n) \ = \ \begin{cases} 
     1+t, & \text{if $n = 1$}, \\
     t, & \text{if $n = 2$}, \\
     0, & \text{if $n \ge 3$} \end{cases} 
\ \ \ \ \ \text{and} \ \ \ \ \ 
\omega_t(n) \ = \ \begin{cases} 
     1, & \text{if $n = 0$}, \\
     t, & \text{if $n = 1$}, \\
     0, & \text{if $n \ge 2$}. 
   \end{cases} \]

\medskip
\noindent
Equation~(\ref{eq:chi1}) is a direct consequence 
of~(\ref{eq:chi}) and this claim. To verify the 
claim, note that the defining equation for 
$\chi_t(n)$ can be rewritten as 
\[ \chi_t(n) \ = \ \sum_{(a_1, a_2,\dots,a_k) 
   \vDash n} (-1)^{n-k} \, [a_1 + 1]_t \, [a_2 + 1]_t 
   \cdots [a_k + 1]_t, \]
where the sum ranges over all sequences 
(compositions) $(a_1, a_2,\dots,a_k)$ of positive 
integers summing to $n$. We leave it as a simple 
combinatorial exercise for the interested reader
to show (for instance, by standard generating 
function methods) that $\chi_t(n) = 0$ for 
$n \ge 3$. The claim follows from this fact and 
the obvious recurrence $\omega_t(n) = \chi_t(n) - 
\omega_t(n-1)$, valid for $n \ge 1$.

Finally, note that Equation~(\ref{eq:chi2}) is 
equivalent to~(\ref{eq:chi1}) in the case $1 \in 
S$. Otherwise, the terms in the left-hand side 
can be partitioned into pairs of terms, 
corresponding to pairs $\{T, T \cup \{1\}\}$ of 
subsets with $1 \not\in T$, canceling each other.
This shows that the left-hand side vanishes.
\qed

\section{Symmetric function identities}
\label{sec:app}

This section derives Equations~(\ref{eq:Ge2}) 
-- (\ref{eq:Ath1-}) from Theorem~\ref{thm:main} 
(Corollaries~\ref{cor:Gessel},~\ref{cor:Ath2} 
and~\ref{cor:Ath1}) and interprets combinatorially
the Schur-positive symmetric functions which 
appear there. We first explain why Gessel's 
identities are special cases of this theorem. 
The set of ascents of a permutation $w \in 
\fS_n$ is defined as $\Asc(w) := [n-1] \sm 
\Des(w)$ and, similarly, we have $\Asc(\pP) 
:= [n-1] \sm \Des(\pP)$ for every standard 
Young tableau $\pP$ of size $n$. Let us recall 
the fact (used in the following proof) that 
the reduced homology groups of a poset with a 
minimum or maximum element vanish. 

\begin{corollary} \label{cor:Gessel}
Equations~(\ref{eq:Ge2}) and~(\ref{eq:Ge1}) are 
valid for the functions 

\begin{equation} \label{eq:corGes1}
  \xi_{n,k} (\bx) \ = \ \sum_{\lambda \vdash n} 
    c_{\lambda,k} \cdot s_\lambda (\bx) \ = \ 
    \sum_{w} F_{n,\Des(w)}(\bx)
\end{equation}

\medskip
\noindent
and

\begin{equation} \label{eq:corGes2}
  \gamma_{n,k} (\bx) \ = \ \sum_{\lambda \vdash n} 
    d_{\lambda,k} \cdot s_\lambda (\bx) \ = \ 
    \sum_{w} F_{n,\Des(w)}(\bx), 
\end{equation}

\medskip
\noindent
where $c_{\lambda,k}$ ${\rm (}$respectively, 
$d_{\lambda,k}{\rm )}$ stands for the number of 
tableaux $\pP \in \SYT(\lambda)$ for which $\Asc(\pP) 
\in \stab([2,n-2])$ ${\rm (}$respectively, 
$\Asc(\pP) \in \stab([n-2]){\rm )}$ has $k$ 
elements and, similarly, $w \in \fS_n$ runs through 
all permutations for which $\Asc(w^{-1}) \in 
\stab([2,n-2])$ ${\rm (}$respectively, $\Asc(w^{-1}) 
\in \stab([n-2]){\rm )}$ has $k$ elements.
\end{corollary}

\begin{proof}
We will apply Theorem~\ref{thm:main} when $P^-$ is the 
Boolean lattice $B_n$, considered as an $\fS_n$-poset 
as in Example~\ref{ex:uniform}. On the one hand, we have 
the equality
\[ 1 \, + \, \sum_{n \ge 2} \, 
   \ch \left( \widetilde{H}_{n-1} 
   ((B_n \sm \{\varnothing\}) \ast T_{t,n-1}; 
   \CC) \right) z^n 
   \ = \ \frac{1-t} {E(\bx; tz) - tE(\bx; z)} \]

\smallskip
\noindent
which, although not explicitly stated in~\cite{SW09},
follows as in the proof of its special case $t=1$ 
\cite[Corollary~5.2]{SW09}. On the other hand,
since $B_n$ has a maximum element, the second summand
in the right-hand side of Equation~(\ref{eq:main2})
vanishes and hence this equation gives
\[ \ch \left( \widetilde{H}_{n-1} 
   ((B_n \sm \{\varnothing\}) \ast T_{t,n-1}; \CC) 
   \right) \ = \sum_{S \in \stab ([2,n-2])} 
   \ch \left( \beta_{B_n} ([n-1] \sm S) \right) \,
   t^{|S|+1} \, (1+t)^{n-2-2|S|} \]

\smallskip
\noindent
for $n \ge 2$. The representations $\beta_{B_n} (S)$
for $S \subseteq [n-1]$ are known to satisfy (see, 
for instance, \cite[Theorem~4.3]{Sta82}) 
\[ \ch \left( \beta_{B_n} (S) \right) \ = \ 
   \sum_{\lambda \vdash n} c_{\lambda,S} \cdot 
   s_\lambda (\bx), \]

\smallskip
\noindent
where $c_{\lambda,S}$ is the number of standard 
Young tableaux of shape $\lambda$ and descent set 
equal to $S$. Combining the previous three 
equalities yields the first equality in 
Equation~(\ref{eq:corGes1}). The second equality 
follows from the first by expanding $s_\lambda(\bx)$ 
according to Equation~(\ref{eq:sFexpan}) to get
\[ \xi_{n,k} (\bx) \ = \ \sum_{\lambda \vdash n} 
   \sum_\pP \sum_{Q \in \SYT (\lambda)} 
   F_{n, \Des(Q)} (\bx) \]

\smallskip
\noindent
where, in the inner sum, $\pP$ runs through all 
tableaux in $\SYT(\lambda)$ for which $\Asc(\pP) \in 
\pP_\stab([2,n-2])$ has $k$ elements, and then using 
the Robinson--Schensted correspondence and its 
standard properties $\Des(w) = \Des(Q(w))$ and 
$\Des(w^{-1}) = \Des(\pP(w))$ to replace the 
summations with one running over elements of $\fS_n$,
as in the statement of the corollary. 

The proof of~(\ref{eq:corGes2}) is entirely 
similar; one has to use Equation~(\ref{eq:main1}) 
instead of~(\ref{eq:main2}), as well as the equality
\[ 1 \, + \, \sum_{n \ge 1} \, 
   \ch \left( \widetilde{H}_{n-1} 
   ((B_n \ast T_{t,n})_{-}; \CC) \right) z^n \ = \ 
   \frac{(1-t) E(\bx; tz)} {E(\bx; tz) - tE(\bx; z)}. \]

\smallskip
\noindent
The latter follows from the proof of Equation~(3.3) 
in~\cite[pp.~15--16]{SW09}, where the left-hand side 
is equal to $-F_t(-z)$, in the notation used in that 
proof. 
\end{proof}

\begin{example} \rm
The coefficient of $z^4$ in the left-hand sides 
of Equations~(\ref{eq:Ge2}) and~(\ref{eq:Ge1}) 
equals

\smallskip
\begin{itemize}
\itemsep=3pt
\item[$\bullet$]
$e_4(\bx) \, (t + t^2 + t^3) + e_2(\bx)^2 \, t^2$, 
and

\item[$\bullet$]
$e_4(\bx) \, (t + t^2 + t^3 + t^4) + e_1(\bx) 
e_3(\bx) \, (t^2 + t^3) + e_2(\bx)^2 \, (t^2 + t^3)$,
\end{itemize}

\smallskip
\noindent
respectively. These expressions may be rewritten as

\smallskip
\begin{itemize}
\itemsep=3pt
\item[$\bullet$]
$s_{(1,1,1,1)} (\bx) \, t(1+t)^2 + s_{(2,1,1)} (\bx) 
\, t^2 + s_{(2,2)} (\bx) \, t^2$, and

\item[$\bullet$]
$s_{(1,1,1,1)} (\bx) \, t(1+t)^3 + 2 s_{(2,1,1)} 
(\bx) \, t^2(1+t) + s_{(2,2)} (\bx) \, t^2(1+t)$,
\end{itemize}

\smallskip
\noindent
respectively, and hence $\xi_{4,0} (\bx) = 
s_{(1,1,1,1)} (\bx)$, $\xi_{4,1} (\bx) = s_{(2,1,1)} 
(\bx) + s_{(2,2)} (\bx)$, $\gamma_{4,0} (\bx) = 
s_{(1,1,1,1)} (\bx)$ and $\gamma_{4,1} (\bx) = 2 
s_{(2,1,1)} (\bx) + s_{(2,2)} (\bx)$. We leave it to 
the reader to verify that these formulas agree with 
Corollary~\ref{cor:Gessel}.
\qed
\end{example}

We now focus on the identities~(\ref{eq:Ath2+})
-- (\ref{eq:Ath1-}). We will apply 
Theorem~\ref{thm:main} to the collection $sB_n$ 
of all subsets of $\{1, 2,\dots,n\} \cup \{ -1, 
-2,\dots,-n\}$ which do not contain $\{i, -i\}$ 
for any index $i$, partially ordered by 
inclusion. This signed analogue of the Boolean 
algebra $B_n$ is a graded poset of rank $n$, 
having the empty set as its minimum element, on 
which the hyperoctahedral group $\bB_n$ acts in
the obvious way \cite[Section~6]{Sta82}, turning 
it into a $\bB_n$-poset. It is isomorphic to the 
poset of faces (including the empty one) of the 
boundary complex of the $n$-dimensional 
cross-polytope and hence it is Cohen--Macaulay 
over $\ZZ$ and any field. The left-hand sides of 
Equations~(\ref{eq:main1}) and~(\ref{eq:main2}) 
for $P^{-} = sB_n$ will be computed using the 
methods of~\cite{SW09}. 

Consider the $n$-element chain $C_n = \{0, 
1,\dots,n-1\}$, with the usual total order. 
Following~\cite{SW09}, we denote by $I_j(B_n)$ the 
order ideal of elements of $(B_n \sm \{\varnothing\})
\ast C_n$ which are strictly less than $([n], j)$. 
Then $I_j(B_n)$ is an $\fS_n$-poset for every $j \in 
C_n$ and one of the main results of~\cite{SW09} (see
\cite[p.~21]{SW09} \cite[Equation~(2.5)]{SW16b}) 
states that

\begin{equation} \label{eq:Ij(Bn)}
   1 \, + \, \sum_{n \ge 1} \, z^n \, 
   \sum_{j=0}^{n-1} \, t^j \, \ch \left(
   \widetilde{H}_{n-2} (I_j(B_n); \CC) \right) 
   \ = \ \frac{(1-t) E(\bx; z)}{E(\bx; tz) - tE(\bx; z)}.
\end{equation}

\smallskip
\begin{proposition} \label{prop:Ath2}
For the $\bB_n$-poset $sB_n$ we have

\begin{equation} \label{eq:prop2}
  1 \, + \, \sum_{n \ge 1} \, \ch_\bB   
  \left( \widetilde{H}_{n-1} ((sB_n \sm \{\varnothing\}) 
  \ast T_{t,n-1}; \CC) \right) z^n 
  \ = \ \frac{(1-t) E(\by; z)} 
        {E(\bx; tz) E(\by; tz) - tE(\bx; z) E(\by; z)}. 
\end{equation}
\end{proposition}

\smallskip
\begin{proof}
Following the reasoning in the proof 
of~\cite[Corollary~5.2]{SW09}, we apply~(\ref{eq:sun})
to the Cohen--Macaulay $\bB_n$-poset obtained from 
$(sB_n \sm \{\varnothing\}) \ast T_{t,n-1}$ by adding
a minimum and a maximum element. For $0 \le j < k \le 
n$, there are exactly $t^j$ $\bB_n$-orbits of elements 
$x$ of rank $k$ in this poset with second coordinate 
of rank $j$ in $T_{t,n-1}$ and for each one of these, 
the open interval $(\hat{0}, x)$ is isomorphic to $I_j(B_k)$ 
and the stabilizer of $x$ is isomorphic to $\fS_k 
\times \bB_{n-k}$. We conclude that
\[ \widetilde{H}_{n-1} ((sB_n \sm \{\varnothing\}) 
   \ast T_{t,n-1}; \CC) \ \cong_{\bB_n} \ 
   \bigoplus_{k=0}^n \, (-1)^{n-k} \, 
   \bigoplus_{j=0}^{k-1} \, t^j \left(
   \widetilde{H}_{k-2} (I_j(B_k); \CC) \otimes 
   1_{\bB_{n-k}} \right)
   \uparrow^{\bB_n}_{\fS_k \times \bB_{n-k}.} \]

\smallskip
\noindent
Applying the map $\ch_\bB$ and 
using the transitivity $\uparrow^{\bB_n}_{\fS_k 
\times \bB_{n-k}} \, \cong_{\bB_n} \, \uparrow^{\bB_k 
\times \bB_{n-k}}_{\fS_k \times \bB_{n-k}} \ 
\uparrow^{\bB_n}_{\bB_k \times \bB_{n-k}}$ of 
induction and properties of $\ch_\bB$ discussed in 
Section~\ref{sec:pre}, the right-hand side becomes
\[ \sum_{k=0}^n \, (-1)^{n-k} \, 
   \sum_{j=0}^{k-1} \, t^j \, \ch \left(
   \widetilde{H}_{k-2} (I_j(B_k); \CC) \right)
   (\bx, \by) \cdot h_{n-k}(\bx). \]

\smallskip
\noindent
Thus, the left-hand side of Equation~(\ref{eq:prop2})
is equal to
\[ H(\bx; -z) \cdot \left( \, 1 \, + \, \sum_{n \ge 1} 
   \, z^n \, \sum_{j=0}^{n-1} \, t^j \, \ch \left(
   \widetilde{H}_{n-2} (I_j(B_n); \CC) \right)
   (\bx, \by) \, \right) \]

\smallskip
\noindent
and the result follows from Equation~(\ref{eq:Ij(Bn)}) 
and the identities $E(\bx; z) H(\bx; -z) = 1$ and 
$E(\bx, \by; z) = E(\bx; z) E(\by; z)$.
\end{proof}

\smallskip
Recall the definition of the sets $\Des(w)$ and 
$\Des(\pP)$ for a signed permutation $w \in \bB_n$ 
and standard Young bitableau $\pP$ of size $n$, 
respectively, from Section~\ref{sec:pre}. Following
\cite[Section~6]{Sta82}, we define the \emph{type 
$B$ descent set} of $\pP = (\pP^+, \pP^{-})$ as 
$\Des_B(\pP) = \Des(\pP) \cup \{n\}$, if $n$ appears
in $\pP^+$, and $\Des_B(\pP) = \Des(\pP)$ otherwise.
The complement of $\Des_B(\pP)$ in the set $[n]$ is 
called the \emph{type $B$ ascent set} of $\pP$ and 
is denoted by $\Asc_B(\pP)$. Similarly, we define 
the \emph{type $B$ descent set} of $w \in \bB_n$ 
as $\Des_B(w) = \Des(w) \cup \{n\}$, if $w(n)$ is
positive, and $\Des_B(w) = \Des(w)$ otherwise. The 
complement of $\Des_B(w)$ in the set $[n]$ is 
called the \emph{type $B$ ascent set} of $w$ and 
is denoted by $\Asc_B(w)$. The sets $\Des_B(w)$ and
$\Des_B(\pP)$ depend only on the signed descent sets 
$\sDes_B(w)$ and $\sDes_B(\pP)$, respectively, and  
\cite[Proposition~5.1]{AAER17}, mentioned at the 
end of Section~\ref{sec:pre}, implies that $\Des_B(w) 
= \Des_B(Q^B(w))$ and $\Des_B(w^{-1}) = \Des_B
(\pP^B(w))$ for every $w \in \bB_n$.

\begin{corollary} \label{cor:Ath2}
Equations~(\ref{eq:Ath2+}) and~(\ref{eq:Ath2-}) are 
valid for the functions 

\begin{equation} \label{eq:corAth2+}
  \xi^+_{n,k} (\bx, \by) \ = \ 
  \sum_{(\lambda, \mu) \vdash n} 
  c^+_{(\lambda,\mu),k} \cdot s_\lambda (\bx) s_\mu 
  (\by) \ = \ \sum_{w} F_{\sDes(w)}(\bx, \by) 
\end{equation}

\medskip
\noindent
and

\begin{equation} \label{eq:corAth2-}
  \xi^{-}_{n,k} (\bx, \by) \ = \ 
  \sum_{(\lambda, \mu) \vdash n} 
  c^{-}_{(\lambda,\mu),k} \cdot s_\lambda (\bx) s_\mu 
  (\by) \ = \ \sum_{w} F_{\sDes(w)}(\bx, \by), 
\end{equation}

\medskip
\noindent
where $c^+_{(\lambda,\mu),k}$ ${\rm (}$respectively, 
$c^{-}_{(\lambda,\mu),k}{\rm )}$ stands for the number 
of bitableaux $\pP \in \SYT(\lambda,\mu)$ for which 
$\Asc_B(\pP) \in \stab([2,n])$ has $k$ elements and 
contains ${\rm (}$respectively, does not 
contain${\rm )}$ $n$ and where, similarly, $w \in 
\bB_n$ runs through all signed permutations for which 
$\Asc_B(w^{-1}) \in \stab([2,n])$ has $k$ 
elements and contains ${\rm (}$respectively, does not 
contain${\rm )}$ $n$.
\end{corollary}

\begin{proof}
We apply the second part of Theorem~\ref{thm:main} 
for $P^- = sB_n$, thought of as a $\bB_n$-poset. 
The representations $\beta_P (S)$ for $S \subseteq 
[n]$ were computed in this case 
in~\cite[Theorem~6.4]{Sta82}, which implies that 

\[ \ch_\bB \left( \beta_{B_n} (S) \right) \ = \ 
   \sum_{(\lambda,\mu) \vdash n} c_{(\lambda,\mu),S} 
   \cdot s_\lambda (\by) s_\mu(\bx) \]

\smallskip
\noindent
for $S \subseteq [n]$, where $c_{(\lambda,\mu),S}$ 
is the number of standard Young bitableaux $\pP$ of 
shape $(\lambda,\mu)$ such that $\Des_B(\pP) = S$. 
Switching the roles of $\bx$ and $\by$ and 
combining this result with the second part of 
Theorem~\ref{thm:main} and 
Proposition~\ref{prop:Ath2} we get
  \begin{eqnarray} \nonumber 
    {\displaystyle \frac{(1-t) E(\bx; z)} 
    {E(\bx; tz) E(\by; tz) - tE(\bx; z) E(\by; z)}}
  & = & 
    \sum_{n \ge 0} z^n \,
    \sum_{k=0}^{\lfloor n/2 \rfloor} \
    \xi^+_{n,k}(\bx, \by) \, t^k (1 + t)^{n-2k}
    \ \ + \\
  & & \nonumber \\
  & & \sum_{n \ge 1} z^n
    \sum_{k=0}^{\lfloor (n-1)/2 \rfloor} 
    \xi^{-}_{n,k}(\bx, \by) \, t^k (1 + t)^{n-1-2k}, 
    \label{eq:sum}
  \end{eqnarray}

\noindent
where the $\xi^{\pm}_{n,k}(\bx, \by)$ are given 
by the first equalities in~(\ref{eq:corAth2+}) 
and~(\ref{eq:corAth2-}). We now note that the 
left-hand side of Equation~(\ref{eq:sum}) is equal
to the sum of the left-hand sides, say
$\Xi^+(\bx, \by, t; z)$ and $\Xi^{-}(\bx, \by, t; 
z)$, of Equations~(\ref{eq:Ath2+}) 
and~(\ref{eq:Ath2-}). Since, as one can readily 
verify, $\Xi^+(\bx, \by, t; z)$ is left invariant
under replacing $t$ with $1/t$ and $z$ with $tz$,
while $\Xi^{-}(\bx, \by, t; z)$ is multiplied by
$t$ after these substitutions, the coefficient of
$z^n$ in $\Xi^+(\bx, \by, t; z)$ (respectively,
$\Xi^{-}(\bx, \by, t; z)$) is a symmetric 
polynomial in $t$ with center of symmetry $n/2$ 
(respectively, $(n-1)/2$) for every $n \in \NN$.
Since the corresponding properties are clear for
the coefficient of $z^n$ in the two summands in
the right-hand side of Equation~(\ref{eq:sum}) 
and because of the uniqueness of the decomposition
of a polynomial $f(t)$ as a sum of two symmetric
polynomials with centers of symmetry $n/2$ and 
$(n-1)/2$ (see \cite[Section~5.1]{Ath17}), we 
conclude that~(\ref{eq:Ath2+}) 
and~(\ref{eq:Ath2-}) follow from the single 
equation~(\ref{eq:sum}). 

The second equalities in~(\ref{eq:corAth2+}) 
and~(\ref{eq:corAth2-}) follow by expanding 
$s_\lambda(\bx) s_\mu(\by)$ according to 
Equation~(\ref{eq:sxyFexpan}) and then using the 
Robinson--Schensted correspondence of type $B$ and 
its properties $\sDes(w) = \sDes(Q^B(w))$ and 
$\Des_B(w^{-1}) = \Des_B(\pP^B(w))$, exactly as 
in the proof of Corollary~\ref{cor:Gessel}.
\end{proof}

\begin{example} \rm
The coefficient of $z^2$ in the left-hand side 
of Equations~(\ref{eq:Ath2+}) and~(\ref{eq:Ath2-}) 
equals

\smallskip
\begin{itemize}
\itemsep=3pt
\item[$\bullet$]
$e_1(\bx) e_1(\by) \, t + e_2(\by) \, t = s_{(1)}
(\bx) s_{(1)}(\by) \, t + s_{(1,1)}(\by) \, t$, and

\item[$\bullet$]
$e_2(\bx) \, (1 + t) = s_{(1,1)}(\bx) \, (1 + t)$,
\end{itemize}

\smallskip
\noindent
respectively and hence $\xi^+_{2,0} (\bx, \by) = 
0$, $\xi^+_{2,1} (\bx, \by) = s_{(1)}(\bx) s_{(1)}(\by) 
+ s_{(1,1)}(\by)$ and $\xi^{-}_{2,0} (\bx, \by) = 
s_{(1,1)}(\bx)$, in agreement with 
Corollary~\ref{cor:Ath2}.
\qed
\end{example}

\begin{proposition} \label{prop:Ath1}
For the $\bB_n$-poset $sB_n$ we have

\begin{equation} \label{eq:prop1}
  1 \, + \, \sum_{n \ge 1} \, \ch_\bB   
  \left( \widetilde{H}_{n-1} 
  ((sB_n \ast T_{t,n})_{-}; \CC) \right) z^n \ = \ 
  \frac{(1-t) E(\by; z) E(\bx; tz) E(\by; tz)} 
  {E(\bx; tz) E(\by; tz) - tE(\bx; z) E(\by; z)}. 
\end{equation}
\end{proposition}

\smallskip
\begin{proof}
Following the reasoning in the proof 
of~\cite[Equation~(3.3)]{SW09}, we set 

\[ L_n (\bx, \by; t) \ := \ \ch_\bB \left(
   L((sB_n \ast T_{t,n})_{-}; \bB_n) \right), \]

\bigskip
\noindent
where $L(P; G)$ denotes the Lefschetz character of 
the $G$-poset $P$ over $\CC$ (see 
Section~\ref{sec:pre}). Since $(sB_n \ast 
T_{t,n})_{-}$ is Cohen--Macaulay over $\CC$ of rank 
$n-1$, we have 

\[ \ch_\bB \left( \widetilde{H}_{n-1} ((sB_n \ast 
   T_{t,n})_{-}; \CC) \right) \ = \ (-1)^{n-1} \, 
   L_n(\bx, \by; t). \]

\medskip
\noindent
Thus, the left-hand side of~(\ref{eq:prop1}) is 
equal to $-\sum_{n \ge 0} L_n(\bx, \by; t) (-z)^n$. 
The sequence of posets $(sB_0, sB_1,\dots,sB_n)$ can 
easily be verified to be $(\bB_0 \times \fS_n, \bB_1 
\times \fS_{n-1},\dots,\bB_n \times \fS_0)$-uniform 
(see Section~\ref{sec:pre}). Moreover, there
is a single $\bB_n$-orbit of elements of $sB_n$ of
rank $k$ for each $k \in \{0, 1,\dots,n\}$.
Thus, applying~(\ref{eq:uniform}) to this 
sequence gives
\[ 1_{\bB_n} \oplus \, \bigoplus_{k=0}^n \, [k + 1]_t 
   \, L((sB_{n-k} \ast T_{t,n-k})_{-}; \bB_{n-k} 
   \times \fS_k) \uparrow^{\bB_n}_{\bB_{n-k} \times 
   \fS_k} \ \, \cong_{\bB_n} \ 0. \]

\smallskip
\noindent
Applying the characteristic map $\ch_\bB$, as in 
the proof of Proposition~\ref{prop:Ath2}, gives
\[ \sum_{k=0}^n \, [k + 1]_t \, h_k(\bx, \by) \, 
   L_{n-k} (\bx, \by; t) \ = \ - h_n (\bx). \]

\smallskip
\noindent
Standard manipulation with generating functions, 
as in the proof of~\cite[Equation~(3.3)]{SW09}, 
results in the formula
\[ \sum_{n \ge 0} L_n(\bx, \by; t) z^n \ = \ - \
   \frac{H(\bx; z)}{\sum_{n \ge 0} \, [n+1]_t \,
   h_n(\bx, \by) z^n} \ = \ - \ \frac{(1-t) H(\bx; z)}
   {H(\bx, \by; z) - t H(\bx, \by; tz)}. \]

\smallskip
\noindent
The proof now follows by switching $z$ to $-z$ and 
using the identities $E(\bx; z) H(\bx; -z) = 1$ and 
$E(\bx, \by; z) = E(\bx; z) E(\by; z)$.
\end{proof}

\begin{corollary} \label{cor:Ath1}
Equations~(\ref{eq:Ath1+}) and~(\ref{eq:Ath1-}) 
are valid for the functions 

\begin{equation} \label{eq:corAth1+}
  \gamma^+_{n,k} (\bx, \by) \ = \ 
  \sum_{(\lambda, \mu) \vdash n} 
  d^+_{(\lambda,\mu),k} \cdot s_\lambda (\bx) s_\mu 
  (\by) \ = \ \sum_{w} F_{\sDes(w)}(\bx, \by) 
\end{equation}

\medskip
\noindent
and

\begin{equation} \label{eq:corAth1-}
  \gamma^{-}_{n,k} (\bx, \by) \ = \ 
  \sum_{(\lambda, \mu) \vdash n} 
  d^{-}_{(\lambda,\mu),k} \cdot s_\lambda (\bx) s_\mu 
  (\by) \ = \ \sum_{w} F_{\sDes(w)}(\bx, \by), 
\end{equation}

\medskip
\noindent
where $d^+_{(\lambda,\mu),k}$ ${\rm (}$respectively, 
$d^{-}_{(\lambda,\mu),k}{\rm )}$ is the 
number of bitableaux $\pP \in \SYT(\lambda,\mu)$ for 
which $\Asc_B(\pP) \in \stab([n])$ has $k$ 
elements and does not contain ${\rm (}$respectively,
contains${\rm )}$ $n$ and, similarly, $w \in \bB_n$ 
runs through all signed permutations for which 
$\Asc_B(w^{-1}) \in \stab([n])$ has $k$ elements 
and does not contain ${\rm (}$respectively,  
contains${\rm )}$ $n$.
\end{corollary}

\begin{proof}
This statement follows by the same reasoning as in 
the proof of Corollary~\ref{cor:Ath2}, provided one 
appeals to the first part of Theorem~\ref{thm:main} 
and Proposition~\ref{prop:Ath1} instead. 
\end{proof}

\begin{example} \rm
The coefficient of $z^2$ in the left-hand side 
of Equations~(\ref{eq:Ath1+}) and~(\ref{eq:Ath1-}) 
equals

\smallskip
\begin{itemize}
\itemsep=3pt
\item[$\bullet$]
$e_2(\bx) \, (1 + t + t^2) + e_1(\bx)^2 \, t + 
e_1(\bx) e_1(\by) \, t = s_{(1,1)}(\bx) \, (1 + t)^2
+ s_{(2)} (\bx) \, t + s_{(1)}(\bx) s_{(1)}(\by) \, 
t$,

\item[$\bullet$]
$e_1(\bx) e_1(\by) \, (t + t^2) + e_2(\by) \, 
(t + t^2) = s_{(1)} (\bx) s_{(1)}(\by) \, t(1+t) + 
s_{(1,1)}(\by) \, t(1+t)$, 
\end{itemize}

\smallskip
\noindent
respectively and hence we have $\gamma^+_{2,0} (\bx, 
\by) = s_{(1,1)}(\bx)$, $\gamma^+_{2,1} (\bx, \by) = 
s_{(2)}(\bx) + s_{(1)}(\bx) s_{(1)}(\by)$ and 
$\gamma^{-}_{2,1} (\bx, \by) = s_{(1)}(\bx) s_{(1)}
(\by) + s_{(2)}(\by)$, in agreement with 
Corollary~\ref{cor:Ath1}.
\qed
\end{example}

\section{An instance of the local equivariant Gal 
phenomenon} \label{sec:face}

This section uses Equation~(\ref{eq:Ath2+}) to 
verify an equivariant analogue of Gal's 
conjecture~\cite{Ga05} for the local face module 
of a certain triangulation of the simplex with 
interesting combinatorial properties. Background 
and any undefined terminology on simplicial 
complexes can be found in~\cite{StaCCA}.

To explain the setup, let $V_n = \{ \varepsilon_1, 
\varepsilon_2,\dots,\varepsilon_n \}$ be the set
of unit coordinate vectors in $\RR^n$ and 
$\Sigma_n$ be the geometric simplex on the vertex 
set $V_n$. Consider a triangulation $\Gamma$ of 
$\Sigma_n$ (meaning, a geometric simplicial 
complex which subdivides $\Sigma_n$) with 
vertex set $V_\Gamma$ and the polynomial ring 
$S = \CC[x_v: v \in V_\Gamma]$ in indeterminates 
which are in one-to-one correspondence with the 
vertices of $\Gamma$. The \emph{face ring} 
\cite[Chapter~II]{StaCCA} of $\Gamma$ is defined 
as the quotient ring $\CC[\Gamma] = S/I_\Gamma$,
where $I_\Gamma$ is the ideal of $S$ generated by 
the square-free monomials which correspond to the
non-faces of $\Gamma$. Following 
\cite[p.~824]{Sta92}, we consider the linear forms
\begin{equation} \label{eq:lsop}
   \theta_i \ = \ \sum_{v \in V_\Gamma} \langle 
   v, \varepsilon_i \rangle x_v 
\end{equation}
for $i \in [n]$, where $\langle \ , \, \rangle$ 
is the standard inner product on $\RR^n$, and 
denote by $\Theta$ the ideal in $\CC[\Gamma]$ 
generated by $\theta_1,\theta_2,\dots,\theta_n$. 
This sequence is a special linear system of 
parameters for $\CC[\Gamma]$, in the sense of 
\cite[Definition~4.2]{Sta92}. As a result, the 
quotient ring $\CC(\Gamma) = \CC[\Gamma] / \Theta$ 
is a finite dimensional, graded $\CC$-vector space 
and so is the \emph{local face module} $L_{V_n} 
(\Gamma)$, defined 
\cite[Definition~4.5]{Sta92} as the image in 
$\CC(\Gamma)$ of the ideal of $\CC[\Gamma]$ 
generated by the square-free monomials which 
correspond to the faces of $\Gamma$ lying in the
interior of $\Sigma_n$. The Hilbert polynomials 
$\sum_{i=0}^n \dim_\CC (\CC(\Gamma))_i t^i$ and
$\sum_{i=0}^n \dim_\CC (L_{V_n}(\Gamma))_i t^i$ 
of $\CC(\Gamma)$ and $L_{V_n} (\Gamma)$ are two 
important enumerative invariants of $\Gamma$,
namely the $h$-polynomial \cite[Section~II.2]
{StaCCA} and the local $h$-polynomial 
\cite[Section~2]{Sta92} \cite[Section~III.10]
{StaCCA}, respectively.  

Suppose that $G$ is a subgroup of the 
automorphism group $\fS_n$ of $\Sigma_n$ which
acts simplicially on $\Gamma$. Then, $G$ acts on 
the polynomial ring $S$ and (as discussed on 
\cite[p.~250]{Ste94}) leaves the $\CC$-linear 
span of $\theta_1, \theta_2,\dots,\theta_n$ 
invariant. Therefore, $G$ acts on the 
graded $\CC$-vector spaces $\CC(\Gamma)$ 
and $L_{V_n} (\Gamma)$ as well and the 
polynomials $\sum_{i=0}^n (\CC(\Gamma))_i t^i$ 
and $\sum_{i=0}^n (L_{V_n}(\Gamma))_i t^i$, whose 
coefficients lie in the representation ring of
$G$, are equivariant generalizations of the 
$h$-polynomial and local $h$-polynomial of 
$\Gamma$, respectively. The pair $(\Gamma, G)$
is said (see also \cite[Section~5.2]{Ath17}) to 
satisfy the \emph{local equivariant Gal 
phenomenon} if 
\begin{equation} \label{eq:Gal}
  \sum_{i=0}^n \, (L_{V_n}(\Gamma))_i \, t^i \ = \ 
  \sum_{k=0}^{\lfloor n/2 \rfloor} M_k \, t^k 
  (1+t)^{n-2k} 
\end{equation}
for some non-virtual $G$-representations $M_k$.
This is an analogue for local face modules of 
the equivariant Gal phenomenon, formulated by 
Shareshian and Wachs~\cite[Section~5]{SW17} for
group actions on (flag) triangulations of spheres
as an equivariant version of Gal's conjecture
\cite[Conjecture~2.1.7]{Ga05}. For trivial
actions on flag triangulations of simplices, 
the validity of the local equivariant Gal 
phenomenon was conjectured in~\cite{Ath12} and 
has been verified in many special cases; see 
\cite[Section~4]{Ath16a} \cite[Section~3.2]{Ath17}
and references therein. 

Although it would be too optimistic to expect
that the local equivariant Gal phenomenon holds 
for all group actions on flag triangulations of 
$\Sigma_n$, the case $G = \fS_n$ deserves special 
attention. We then use the notation
\begin{eqnarray*}
\ch \left( \CC(\Gamma), t \right) & := & 
\sum_{i=0}^n \, \ch \left( \CC(\Gamma) \right)_i 
  t^i, \\
\ch \left( L_{V_n}(\Gamma), t \right) & := & 
\sum_{i=0}^n \, \ch \left( L_{V_n}(\Gamma) \right)_i 
  t^i. 
\end{eqnarray*}
For the (first) barycentric subdivision of 
$\Sigma_n$ we have the following result of Stanley.
\begin{proposition} \label{prop:sdn}
{\rm (\cite[Proposition~4.20]{Sta92})} 
For the $\fS_n$-action on the barycentric 
subdivision $\Gamma_n$ of the simplex $\Sigma_n$, 
we have
\begin{equation} \label{eq:Lsdn}
  1 \, + \, \sum_{n \ge 1} \, \ch 
  \left( L_{V_n}(\Gamma_n), t \right) z^n \ = \ 
  \frac{1 - t}{H(\bx; tz) - tH(\bx; z)}.
\end{equation}
\end{proposition}

Combining this result with Gessel's identity 
(\ref{eq:Ge2}) gives 
\[ \ch \left( L_{V_n}(\Gamma_n), t \right) \ = \ 
   \sum_{k=0}^{\lfloor (n-2)/2 \rfloor} \omega \,
   \xi_{n,k}(\bx) \, t^{k+1} (1 + t)^{n-2k-2}, \]
where $\omega$ is the standard involution on 
symmetric functions exchanging $e_\lambda(\bx)$ 
and $h_\lambda(\bx)$ for every $\lambda$, whence 
it follows that $(\Gamma_n, \fS_n)$ satisfies the 
local equivariant Gal phenomenon for every $n$.

The combinatorics of the barycentric subdivision 
$\Gamma_n$ is related to the symmetric group $\fS_n$.
We now consider a triangulation $K_n$ of the simplex 
$\Sigma_n$, studied in \cite[Chapter~3]{Sav13} (see 
also \cite[Remark~4.5]{Ath16a} 
\cite[Section~3.3]{Ath17}) and shown on the right of 
Figure~\ref{fig:K3} for $n=3$, the combinatorics of 
which is related to the hyperoctahedral group $\bB_n$.
The triangulation $K_n$ can be defined as the 
barycentric subdivision of the standard cubical 
subdivision of $\Sigma_n$, shown on the left of 
Figure~\ref{fig:K3} for $n=3$, whose faces are in 
inclusion-preserving bijection with the nonempty 
closed intervals in the truncated Boolean lattice 
$B_n \sm \{ \varnothing \}$. Thus, the faces of 
$K_n$ correspond bijectively to chains of nonempty 
closed intervals in $B_n \sm \{ \varnothing \}$ 
and $\fS_n$ acts simplicially on $K_n$ in the 
obvious way. As a simplicial complex, $K_n$ can be 
thought of as a `half Coxeter complex' for $\bB_n$.

  \begin{figure}
  \epsfysize = 1.4 in 
  \centerline{\epsffile{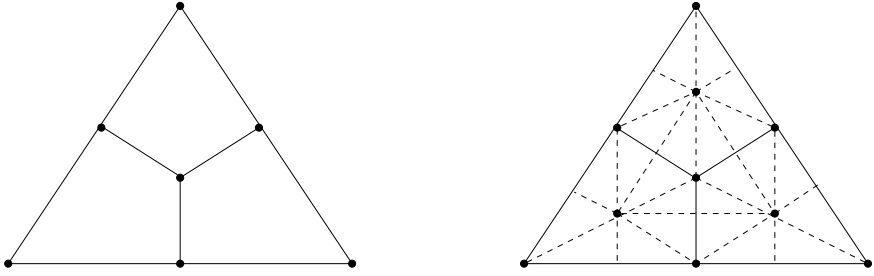}}
  \caption{The triangulation $K_3$}
  \label{fig:K3}
  \end{figure}

\begin{proposition} \label{prop:Kn}
For the $\fS_n$-action on $K_n$ we have
\begin{equation} \label{eq:Kn}
  1 \, + \, \sum_{n \ge 1} \, \ch 
  \left( \CC(K_n), t \right) z^n \ = \ \frac{H(\bx; z)
  \left( H(\bx; tz) - tH(\bx; z) \right)} 
        {H(\bx; tz)^2 - tH(\bx; z)^2}
\end{equation}
and
\begin{equation} \label{eq:localKn}
  1 \, + \, \sum_{n \ge 1} \, \ch 
  \left( L_{V_n}(K_n), t \right) z^n \ = \ 
  \frac{H(\bx; tz) - tH(\bx; z)} 
        {H(\bx; tz)^2 - tH(\bx; z)^2}.
\end{equation}
Moreover, the pair $(K_n, \fS_n)$ satisfies the 
local equivariant Gal phenomenon for every $n$.
\end{proposition}

The proof relies on methods developed by 
Stembridge~\cite{Ste94} to study representations of 
Weyl groups on the cohomology of the toric varieties
associated to Coxeter complexes. To prepare for it, 
we recall that the \emph{$h$-polynomial} of a
simplicial complex $\Delta$ of dimension $n-1$
is defined as 
\[ h(\Delta, t) \ = \ \sum_{i=0}^n f_{i-1} 
  (\Delta) \, t^i (1-t)^{n-i}, \]
where $f_i(\Delta)$ stands for the number of 
$i$-dimensional faces of $\Delta$. Consider  
a pair $(\Gamma, G)$, consisting of a 
triangulation $\Gamma$ of $\Sigma_n$ and a 
subgroup $G$ of $\fS_n$ acting on $\Gamma$, as 
discussed earlier. Following \cite[Section~1]
{Ste94}, we call the action of $G$ 
on $\Gamma$ \emph{proper} if $w$ fixes all 
vertices of every face $F \in \Delta$ which is 
fixed by $w$, for every $w \in G$. Note that 
group actions, such as the $\fS_n$-actions on 
$\Gamma_n$ and $K_n$, on the order complex 
(simplicial complex of chains) of a poset which 
are induced by an action on the poset itself, 
are always proper. Under this assumption, the 
set $\Gamma^w$ of faces of $\Gamma$ which are 
fixed by $w$ forms an induced subcomplex of 
$\Gamma$, for every $w \in G$.

Although Stembridge~\cite{Ste94} deals with 
triangulations of spheres, rather than simplices, 
his methods apply to our setting and his Theorem~1.4,
combined with the considerations of Section~6 in
\cite{Ste94}, imply that  
\begin{equation} \label{eq:ste}
  \ch \left( \CC(\Gamma), t \right) \ = \ 
  \frac{1}{n!} \sum_{w \in \fS_n} 
  \frac{h(\Gamma^w, t)}{(1 - t)^{1+\dim(\Gamma^w)}}
  \, \prod_{i \ge 1} \, (1 - t^{\lambda_i(w)}) \,
  p_{\lambda_i(w)} (\bx)
\end{equation}
for every proper $\fS_n$-action on $\Gamma$, where 
$\lambda_1(w) \ge \lambda_2(w) \ge \cdots$ are the 
sizes of the cycles of $w \in \fS_n$ and $p_k(\bx)$
is a power sum symmetric function.

\medskip
\noindent
\emph{Proof of Proposition~\ref{prop:Kn}}. To prove 
Equation~(\ref{eq:Kn}), we follow the analogous 
computation in \cite[Section~6]{Ste94} for the 
barycentric subdivision of the boundary complex of 
the simplex. We first note that
$(K_n)^w$ is combinatorially isomorphic to 
$K_{c(w)}$ for every $w \in \fS_n$, where $c(w)$ is 
the number of cycles of $w$. Furthermore, it was 
shown in~\cite[Section~3.6]{Sav13} that, in the 
notation of Section~\ref{sec:proof}, $h(K_n, t)$
is the `half $\bB_n$-Eulerian polynomial'
\begin{equation} 
  B^+_n(t) \ = \ \sum_{w \in \bB^+_n} 
   t^{|\Des_B(w)|}, 
\end{equation} 
where $\bB^+_n$ consists of the signed permutations
$w \in \bB_n$ with \emph{negative} first coordinate.
These remarks and Equation~(\ref{eq:ste}) imply 
that 
\[ \ch \left( \CC(K_n), t \right) z^n \ = 
   \sum_{\lambda = (\lambda_1, \lambda_2,\dots) 
   \vdash n} m^{-1}_\lambda \, 
   \frac{B^+_{\ell(\lambda)}(t)}
   {(1 - t)^{\ell(\lambda)}} \, \prod_{i \ge 1} \, 
   (1 - t^{\lambda_i}) \, p_{\lambda_i} (\bx)
   z^{\lambda_i}, \]
where $n!/m_\lambda$ is the cardinality of the 
conjugacy class of $\fS_n$ which corresponds to 
$\lambda \vdash n$ and $\ell(\lambda)$ is the number 
of parts of $\lambda$. The polynomials $B^+_n(t)$
are known (see, for instance, 
\cite[Equation~3.7.5]{Sav13}) to satisfy 
\begin{equation} 
   \frac{B^+_n(t)} {(1 - t)^n} \ = \ \sum_{k \ge 0}
   \left( (2k+1)^n - (2k)^n \right) t^k 
\end{equation} 
and hence, we may rewrite the previous formula as 
\[ \ch \left( \CC(K_n), t \right) z^n \ = 
   \sum_{k \ge 0} \, t^k 
   \sum_{\lambda = (\lambda_1, \lambda_2,\dots) 
   \vdash n} m^{-1}_\lambda 
   \left( (2k+1)^{\ell(\lambda)} - 
   (2k)^{\ell(\lambda)} \right) \prod_{i \ge 1} \, 
   (1 - t^{\lambda_i}) \, p_{\lambda_i} (\bx)
   z^{\lambda_i} . \]
Summing over all $n \ge 1$ and using the standard 
identities
\[ H(\bx; z) \ = \ \sum_\lambda m^{-1}_\lambda 
   p_\lambda (\bx) z^{|\lambda|} \ = \ \exp \left( 
   \, \sum_{n \ge 1} p_n(\bx) z^n / n \right) \]
just as in the proof of \cite[Theorem~6.2]{Ste94}
(one considers the $p_n$ as algebraically 
independent indeterminates and replaces first each
$p_n$ with $(2k+1) (1 - t^n)p_n$, then with $(2k) 
(1 - t^n)p_n$), we conclude that 
\begin{eqnarray*}
  1 \, + \, \sum_{n \ge 1} \, \ch 
  \left( \CC(K_n), t \right) z^n & = & 1 \, + \, 
  \sum_{k \ge 0} \, t^k 
  \left( \frac{H(\bx; z)^{2k+1}}{H(\bx; tz)^{2k+1}} 
  - \frac{H(\bx; z)^{2k}}{H(\bx; tz)^{2k}} \right) 
  \\ & & \\
  & = & 1 \, + \, \left( \frac{H(\bx; z)}{H(\bx; tz)} 
  - 1 \right) \left( 1 - t \, \frac{H(\bx; z)^2}
  {H(\bx; tz)^2} \right)^{-1} \\ & & \\
  & = & \frac{H(\bx; z)
  \left( H(\bx; tz) - tH(\bx; z) \right)} 
        {H(\bx; tz)^2 - tH(\bx; z)^2}
\end{eqnarray*}

\medskip
\noindent
and the proof of~(\ref{eq:Kn}) follows. To prove
(\ref{eq:localKn}), it suffices to observe that
\[ 1 \, + \, \sum_{n \ge 1} \, \ch \left( 
   L_{V_n}(K_n), t \right) z^n \ = \ E(\bx, -z) 
   \left( 1 \, + \, \sum_{n \ge 1} \, \ch \left( 
   \CC(K_n), t \right) z^n \right). \]
The latter follows exactly as the corresponding 
identity for the barycentric subdivision 
$\Gamma_n$, shown in the proof of 
\cite[Proposition~4.20]{Sta92}. Finally, from
Equations~(\ref{eq:Ath2+}) and~(\ref{eq:localKn})
we deduce that
\[ \ch \left( L_{V_n}(K_n), t \right) \ = \ 
   \sum_{k=0}^{\lfloor n/2 \rfloor} \omega \, 
   \xi^+_{n,k} (\bx, \bx) \, t^k (1 + t)^{n-2k}. \]
This expression, Corollary~\ref{cor:Ath2} and the 
well known fact that $s_\lambda(\bx) s_\mu(\bx)$ 
is Schur-positive for all partitions $\lambda, \mu$
imply that $\ch \left( L_{V_n}(K_n), t \right)$ is 
Schur $\gamma$-positive for every $n$, as claimed
in the last statement of the proposition.
\qed

\section{An instance of the equivariant Gal 
phenomenon} \label{sec:toric}

A very interesting group action on a simplicial 
complex is that of a finite Coxeter group $W$ on 
its Coxeter complex~\cite{Bj84}. When $W$ 
is crystallographic, this action induces a graded 
$W$-representation on the (even degree) cohomology of 
the associated projective toric variety which has been 
studied by Procesi~\cite{Pro90}, 
Stanley~\cite[p.~529]{Sta89}, Dolgachev and 
Lunts~\cite{DL94}, Stembridge~\cite{Ste94} and
Lehrer~\cite{Le08}, among others. The graded dimension
of this representation is equal to the $W$-Eulerian 
polynomial. Its equivariant $\gamma$-positivity is
a consequence of a variant of Equation~(\ref{eq:Ge1})
in the case of the symmetric group $\fS_n$; see 
\cite[Section~5]{Ath17} \cite[Section~5]{SW17}. In the 
case of the hyperoctahedral group $\bB_n$, by 
\cite[Theorem~6.3]{DL94} or \cite[Theorem~7.6]{Ste94}, 
the Frobenius characteristic of this graded 
$\bB_n$-representation is equal to the coefficient of 
$z^n$ in 
\begin{equation} \label{eq:genrefEulerB}
  \frac{(1-t) H(\bx; z) H(\bx; tz)} 
  {H(\bx, \by; tz) - tH(\bx, \by; z)} . 
\end{equation}

The following statement (and its proof) shows that the 
equivariant $\gamma$-positivity of this graded 
representation is a consequence of the results of 
Section~\ref{sec:app} and confirms another instance 
of the equivariant Gal phenomenon of Shareshian and 
Wachs. As discussed in \cite[Section~5]{Ath17}, it is 
reasonable to expect that the same holds for the 
action of any finite crystallographic Coxeter group 
$W$ on its Coxeter complex; that would provide a 
natural equivariant analogue to the $\gamma$-positivity 
of $W$-Eulerian polynomials \cite[Section~2.1.3]{Ath17}. 

\begin{proposition}
The coefficient of $z^n$ in (\ref{eq:genrefEulerB}) 
is Schur $\gamma$-positive for every $n \in \NN$.
\end{proposition}

\begin{proof}
Using Equations~~(\ref{eq:Ge2}) with~(\ref{eq:Ath1+})
we find that 

\begin{eqnarray*} 
\frac{(1-t) E(\bx; z) E(\bx; tz)} 
	{E(\bx, \by; tz) - tE(\bx, \by; tz)}
& = & \frac{E(\bx; z) E(\bx; tz)(E(\by; tz) - tE(\by; z))} 
	{E(\bx; tz)E(\by; tz) - tE(\bx; z)E(\by; tz)} \cdot 
\frac{1-t} 
	{E(\by; tz) - tE(\by; z)} \\ & & \\
& = & \left( \, \sum_{n \ge 0} z^n 
\sum_{i=0}^{\lfloor n/2 \rfloor} \gamma^{+}_{n,i}
(\bx, \by) \, t^i (1 + t)^{n-2i} \right) \cdot \\ & & \\
&   & \left( \, 1 \, + \, \sum_{n \ge 2} z^n 
\sum_{j=1}^{\lfloor n/2 \rfloor} 
\xi_{n,j-1}(\by) \, t^j (1 + t)^{n-2j} \right) \\ & & 
\\ & = & \sum_{n \ge 0} z^n \! \sum_{k+\ell=n} 
\sum_{i, j} \, \gamma^{+}_{k,i} (\bx, \by) 
\, \xi_{\ell,j-1}(\by) \,  t^{i+j} (1+t)^{n-2i-2j} ,
\end{eqnarray*}

\medskip
\noindent
where we have set $\xi_{0, -1}(\by) := 1$ and $\xi_{1, -1}
(\by) := 0$. Since Schur-positivity is preserved 
by sums, products and the standard involution on symmetric 
functions, this computation implies that the coefficient of 
$z^n$ in (\ref{eq:genrefEulerB}) is Schur $\gamma$-positive 
for every $n \in \NN$ and the proof follows.
\end{proof}

\begin{remark} \rm
We have shown that
\begin{equation} \label{eq:gammaB(x,y)}
\frac{(1-t) H(\bx; z) H(\bx; tz)} 
{H(\bx, \by; tz) - tH(\bx, \by; z)} \ = \ 1 \ + \ 
\sum_{n \ge 1} z^n \sum_{i=0}^{\lfloor n/2 \rfloor} 
\gamma^B_{n,i}(\bx,\by) \, t^i (1 + t)^{n-2i}
\end{equation}
for some Schur-positive symmetric functions $\gamma^B_{n,i}
(\bx, \by)$. It is an open problem to find an explicit
combinatorial interpretation of the coefficient
$c^B_{(\lambda, \mu), i}$ of $s_\lambda(\bx) s_\mu(\by)$ in 
$\gamma^B_{n,i}(\bx, \by)$, for $(\lambda, \mu) \vdash n$.
Comparing the graded dimensions of the 
$\bB_n$-representations whose Frobenius characteristic is
given by the two sides of Equation~(\ref{eq:gammaB(x,y)}) we 
get
\[ B_n(t) \ = \ \sum_{(\lambda,\mu) \vdash n} 
   {n \choose |\lambda|} f^\lambda f^\mu 
	 \sum_{i=0}^{\lfloor n/2 \rfloor} c^B_{(\lambda,\mu), i} \, 
	 t^i (1 + t)^{n-2i} , \]
where $B_n(t) := \sum_{w \in \bB_n} t^{|\Des_B(w)|}$ is the 
$\bB_n$-Eulerian polynomial and $f^\lambda$ stands for the 
number of standard Young tableaux of shape $\lambda$. Thus, 
a solution to this problem would provide a refinement of the 
known $\gamma$-positivity of $\bB_n$-Eulerian polynomials; 
see \cite[Section~2.1.3]{Ath17}.
\end{remark}

\medskip
\noindent \textbf{Acknowledgements}. 
The author wishes to thank Eric Katz and Kalle Karu
for some helpful e-discussions about the local 
equivariant Gal phenomenon and the anonymous referees
for their useful comments.

\end{document}